\documentclass[a4paper, 10pt]{amsart}
\usepackage[utf8]{inputenc}
\usepackage{amsmath, amssymb, amsthm}
\usepackage{color}
\definecolor{darkblue}{rgb}{0,0,0.6}
\usepackage[ocgcolorlinks]{hyperref}
\usepackage[nameinlink]{cleveref}
\usepackage{tikz}
\usetikzlibrary{matrix,arrows,trees}

\hypersetup{
 colorlinks = true,
 linkcolor = darkblue,
 citecolor = darkblue,
 urlcolor = darkblue,
}

\title[$A$--theoretic Farrell--Jones for solvable groups]{The $A$--theoretic Farrell--Jones Conjecture for virtually solvable groups}

	\author[D. Kasprowski]{Daniel Kasprowski}
	\address{Max Planck Institute for Mathematics, Vivatsgasse 7, 53111 Bonn, Germany}
	\email{kasprowski@mpim-bonn.mpg.de}
	\urladdr{http://people.mpim-bonn.mpg.de/kasprowski/}

	\author[M. Ullmann]{Mark Ullmann}
	\address{FU Berlin, Institut f\"ur Mathematik, Arnimallee 7, 14195 Berlin, Germany}
    \email{mark.ullmann@math.fu-berlin.de}
    \urladdr{http://www.mi.fu-berlin.de/math/groups/top/members/postdoc/ullmann.html}

    \author[C. Wegner]{Christian Wegner}
    \address{Hausdorff Research Institute for Mathematics (HIM), Poppelsdorfer Allee 45, 53115 Bonn, Germany}
    \email{wegner@him.uni-bonn.de}
    \urladdr{http://www.math.uni-bonn.de/people/wegner/}

     \author[C. Winges]{Christoph Winges}
     \address{Hausdorff Research Institute for Mathematics (HIM), Poppelsdorfer Allee 45, 53115 Bonn, Germany}
     \email{christoph.winges@wwu.de}
     \urladdr{http://wwwmath.uni-muenster.de/reine/u/christoph.winges/}


\theoremstyle{plain}
\newtheorem{theorem}{Theorem}[section]

\newtheorem{lemma}[theorem]{Lemma}

\newtheorem{corollary}[theorem]{Corollary}

\newtheorem{proposition}[theorem]{Proposition}

\newtheorem*{theorem*}{Theorem}
\newtheorem*{mtheorem*}{Main Theorem}

\theoremstyle{definition}
\newtheorem{definition}[theorem]{Definition}

\newtheorem{remark}[theorem]{Remark}

\numberwithin{equation}{section}

\renewcommand{\epsilon}{\varepsilon}
\renewcommand{\theta}{\vartheta}
\renewcommand{\phi}{\varphi}

\newcommand{\Aa}{\mathbb{A}}

\newcommand{\FF}{\mathbb{F}}

\newcommand{\JJ}{\mathbb{J}}

\newcommand{\NN}{\mathbb{N}}

\newcommand{\QQ}{\mathbb{Q}}
\newcommand{\RR}{\mathbb{R}}

\newcommand{\ZZ}{\mathbb{Z}}

\newcommand{\cC}{\mathcal{C}}

\newcommand{\cF}{\mathcal{F}}

\newcommand{\cO}{\mathcal{O}}

\newcommand{\cR}{\mathcal{R}}

\newcommand{\cV}{\mathcal{V}}

\newcommand{\frakC}{\mathfrak{C}}

\newcommand{\frakS}{\mathfrak{S}}

\newcommand{\frakZ}{\mathfrak{Z}}

\newcommand{\frakp}{\mathfrak{p}}

\newcommand{\abs}[1]{\lvert #1 \rvert}
\newcommand{\dress}{\mathcal{D}r}
\newcommand{\EGF}[2]{E_{#2}(#1)}

\DeclareMathOperator{\cells}{\diamond}

\DeclareMathOperator{\im}{im}

\newcommand{\id}{\textup{id}}

\newcommand{\trans}{\textup{trans}}

\newcommand{\diam}{\textup{diam}}

\def\N{\mathbb N}

\begin{document}

\begin{abstract}
 We prove the $A$--theoretic Farrell--Jones Conjecture for virtually solvable groups.
 As a corollary, we obtain that the conjecture holds for $S$--arithmetic groups and lattices in almost connected Lie groups.
\end{abstract}

\maketitle

\section{Introduction}

For every group $G$ there is a functor $\Aa \colon \mathrm{Or}G \to \mathrm{Spectra}$ from the orbit category of $G$ to the category of spectra sending $G/H$ to (a spectrum weakly equivalent to) the non-connective $A$-theory spectrum $\Aa(BH)$. For any such functor $\FF \colon \mathrm{Or}G \to \mathrm{Spectra}$, a $G$--homology theory $H_\FF$ can be constructed via
\[H_\FF(X):= \mathrm{Map}_G(\_,X_+) \wedge_{\mathrm{Or}G} \FF,\]
see Davis and Lück \cite{Davis-Lueck(1998)}. We will denote its homotopy groups by $H_n^G(X;\FF):=\pi_nH_\FF(X)$. The assembly map for the family of virtually cyclic subgroups (in $A$--theory) is the map
\begin{equation*}
	H_n^G(E_{\cV\cC yc}G;\Aa)\to H^G_n(\mathrm{pt};\Aa)\cong A_n(BG)
\end{equation*}
induced by the map $E_{\cV\cC yc} G \to \mathrm{pt}$. Here, $E_{\cV\cC yc}G$ denotes the classifying space for the family of virtually cyclic subgroups, see L\"uck \cite{Lueck(2005)}. 
The assembly map can more generally be defined with coefficients, cf.~\cite[Conjecture~7.1]{UW}.
In this note, we consider the \emph{$A$--theoretic Farrell--Jones Conjecture with coefficients and finite wreath products}, which predicts for a discrete group $G$ that the assembly map with coefficients is an isomorphism for every wreath product $G \wr F$ of $G$ with a finite group $F$.

Our main result is the following:

\begin{theorem}\label{maintheorem}
 Let $G$ be a virtually solvable group.
 Then $G$ satisfies the Farrell--Jones Conjecture for $A$--theory with coefficients and finite wreath products.
\end{theorem}

Using this, we can adapt previous work by R\"uping \cite{Rueping(2016)} and Kammeyer, L\"uck and R\"uping \cite{KLR(2016)} to $A$--theory:

\begin{corollary}\label{cor:s-arithmetic}
 The $A$--theoretic Farrell--Jones Conjecture with coefficients and finite wreath products holds for
 subgroups of $\mathrm{GL}_n(\QQ)$ or $\mathrm{GL}_n(F(t))$, where $F$ is a finite field.
 
 In particular, the conjecture holds for $S$--arithmetic groups.
\end{corollary}
\begin{proof}
 The proof works as the one of \cite[Theorem~8.13]{Rueping(2016)}:
 Since the conjecture is inherited under directed colimits \cite[Theorem~1.1(ii)]{ELPUW},
 it suffices to consider linear groups over localizations at finitely many primes.
 Then \cite[Proposition~2.2]{Rueping(2016)} together with \cite[Corollary~6.6]{ELPUW} shows that such a group satisfies the conjecture relative to a certain family of subgroups, all whose members in turn satisfy the conjecture relative to the class of virtually solvable groups \cite[Theorem~8.12]{Rueping(2016)}.
 The corollary follows from \Cref{maintheorem} together with the Transitivity Principle \cite[Proposition~11.2]{UW}.
\end{proof}

\begin{corollary}\label{cor:lattices}
 The $A$--theoretic Farrell--Jones Conjecture with coefficients and finite wreath products holds for
 arbitrary lattices in almost connected Lie groups.
 
 More generally, it holds for lattices $\Gamma$ in second countable, locally compact Hausdorff groups $G$ whose group of path components $\pi_0(G)$ is discrete and satisfies the $A$--theoretic Farrell--Jones Conjecture with coefficients and finite wreath products.
\end{corollary} 
\begin{proof}
 In \cite{KLR(2016)}, it is shown that a class of groups satisfying the list of properties from \cite[Theorem~2]{KLR(2016)} also contains the groups considered in the corollary.

 The statement of \cite[Theorem~2]{KLR(2016)} holds for the class of groups satisfying the $A$--theoretic Farrell--Jones Conjecture with coefficients and finite wreath products by \cite[Theorem~1.1]{ELPUW}, \Cref{maintheorem} and \Cref{cor:s-arithmetic}.
\end{proof}

As explained in \cite[Section~3]{ELPUW}, the analogous statements of \Cref{maintheorem}, \Cref{cor:s-arithmetic} and \Cref{cor:lattices} for (topological, PL or smooth) Whitehead spectra and pseudoisotopy spectra also hold true.

\begin{remark}
 We have been informed that Thomas Farrell and Xiaolei Wu have independently obtained a proof of \Cref{maintheorem}.
\end{remark}

\subsection*{Acknowledgements}
This paper was conceived and written during the Junior Trimester Program ``Topology" at the Hausdorff Research Institute for Mathematics (HIM) in Bonn. The first author was financially supported by the Max Planck Society.

\section{Dress--Farrell--Hsiang--Jones groups}

The proof of the $A$--theoretic Farrell--Jones Conjecture for solvable groups relies on a concoction of the Farrell--Hsiang method \cite{UW} and transfer reducibility \cite{ELPUW} which mimics the combination of the methods from \cite{Bartels-Lueck(2012b)} and \cite{Bartels-Lueck(2012), Wegner(2012)} in \cite{Wegner(2015)}.

\begin{definition}\label{def:dress.groups}
 Let $F$ be a finite group. We call $F$ a \emph{Dress group} if there exists a normal series $P \unlhd H \unlhd F$ such that $P$ is a $p$--group for some prime $p$, $H/P$ is cyclic and $F/H$ is a $q$--group for some prime $q$.
\end{definition}

We refer to \cite[Definition~2.7]{Wegner(2015)} and \cite[Definition~2.12]{Wegner(2015)} for the definitions of ``homotopy coherent $G$--action" and ``controlled domination".

\begin{definition}\label{def:dfhj.group}
 Let $G$ be a discrete group and let $S \subseteq G$ be a finite and symmetric generating set of $G$ which contains the trivial element. Let $\cF$ be a family of subgroups of $G$.
 
 Then $G$ is a \emph{Dress--Farrell--Hsiang--Jones group with respect to $\cF$}, or \emph{DFHJ-group (with respect to $\cF$)} for short,
 if there exists $N \in \NN$  such that for every $n \in \NN$ there is a homomorphism $\pi \colon G \to F$ to a finite group with the property that for every Dress subgroup $D \leq F$ there exist
 \begin{enumerate}
  \item\label{it:dfhj.group1} a compact, contractible metric space $X_D$ such that for every $\epsilon > 0$ there is an $\epsilon$--controlled domination of $X_D$ by an at most $N$--dimensional finite simplicial complex;
  \item\label{it:dfhj.group2} a homotopy coherent $G$--action $\Gamma_D$ on $X_D$;
  \item\label{it:dfhj.group3} a $\pi^{-1}(D)$--simplicial complex $\Sigma_D$ of dimension at most $N$ whose isotropy is contained in $\cF$;
  \item\label{it:dfhj.group4} a $\pi^{-1}(D)$--equivariant map $f_D \colon G \times X_D \to \Sigma_D$ such that
   \begin{itemize}
    \item for all $g \in G$, $x \in X_D$ and $s \in S^n$
        \[
         d^{l^1}\big( f_D(g,x),f_D(gs^{-1},\Gamma_D(s,x)) \big) \leq \frac{1}{n},
        \]
    \item for all $g \in G$, $x \in X_D$ and $s_0,\ldots,s_n \in S^n$
        \[
         \diam\big\{ f_D(g,\Gamma_D(s_n,t_n,\ldots,s_0,x)) \,\big|\, t_1,\ldots,t_n \in [0,1] \big\} \leq \frac{2}{n}.
        \]
   \end{itemize}
 \end{enumerate}
\end{definition}

\begin{remark}\label{rem:dfhj.transfer-reducible.dfh}
 \ 
 \begin{enumerate}
  \item If $G$ is homotopy transfer reducible with respect to $\cF$ \cite[Definition~6.2]{ELPUW}, then it is a DFHJ-group with respect to $\cF$: Choose the finite quotient to be trivial for all $n$.
  \item If $G$ is a Dress--Farrell--Hsiang group with respect to $\cF$ \cite[Definition~7.3]{UW}, then it is a DFHJ-group with respect to $\cF$: Choose the transfer space $X_D$ to be a point for all $n$ and $D$.
 \end{enumerate}
\end{remark}

\begin{remark}\label{rem:comparison.contracting.conditions}
 Condition \eqref{it:dfhj.group4} in \Cref{def:dfhj.group} looks a bit different than \cite[Definition~4.1]{Wegner(2015)}.
 The difference lies mostly in notation.
 As we argue in the proof of \Cref{prop:gw.is.dfhj} below, the condition in \cite[Definition~4.1]{Wegner(2015)} implies ours.
 Conversely, the proof showing the existence of the functor $F$ in diagram~\eqref{eq:outline.of.proof} (cf.~\cite[Lemma~6.11]{ELPUW}) shows that condition~\eqref{it:dfhj.group4} also yields the condition in \cite[Definition~4.1]{Wegner(2015)}, up to some constants. 
\end{remark}

\begin{theorem}\label{thm:afjc.for.dfhjgroups}
 Suppose that $G$ is a DFHJ-group with respect to a family $\cF$ of subgroups of $G$.
 
 Then the $A$--theoretic isomorphism conjecture with coefficients relative $\cF$ holds for $G$.
\end{theorem}

The remainder of this section is dedicated to a proof of \Cref{thm:afjc.for.dfhjgroups} and is modelled on \cite[Section~4.2]{Wegner(2015)}.
Just like the proofs in \cite{UW, ELPUW}, we show that the fiber of the assembly map is weakly contractible.
This uses the fact that this fiber can be modelled by the $K$--theory of certain categories of controlled retractive spaces, whose definition we recall next (cf.~also \cite[Sections~2 and 3]{UW}).

A \emph{coarse structure} is a triple $\frakZ=(Z, \mathfrak{C}, \mathfrak{S})$ such that $Z$ is a Hausdorff $G$--space, $\frakC$ is a collection of reflexive, symmetric and $G$--invariant relations on $Z$ which is closed under taking finite unions and compositions, and $\frakS$ is a collection of $G$--invariant subsets of $Z$ which is closed under taking finite unions. See \cite[Definition~3.23]{UW} for the notion of a \emph{morphism of coarse structures}.
  
Fix a coarse structure $\frakZ$.
  
A \emph{labeled $G$--CW--complex relative W}, see \cite[Definition~2.3]{UW}, is a pair $(Y, \kappa)$, where $Y$ is a free $G$--CW--complex relative $W$ together with a $G$--equivariant function $\kappa \colon \cells Y \rightarrow Z$. Here, $\diamond Y$ denotes the (discrete) set of relative cells of $Y$.
  
A \emph{$\frakZ$--controlled map} $f \colon (Y_1,\kappa_1) \rightarrow (Y_2, \kappa_2)$ is a $G$--equivariant, cellular map $f \colon Y_1 \rightarrow Y_2$ relative $W$ such that for all $k \in \mathbb{N}$ there is some $C \in \mathfrak{C}$ for which
\[
 (\kappa_2,\kappa_1)(\{(e_2,e_1) \mid e_1 \in \cells_k Y_1, e_2\in \cells Y_2, \langle f(e_1) \rangle \cap e_2 \neq \emptyset\}) \subseteq C
\]
holds.
  
A \emph{$\frakZ$--controlled $G$--CW--complex relative W} is a labeled $G$--CW--complex $(Y,\kappa)$ relative $W$, such that the identity is a $\frakZ$--controlled map and for all $k \in \mathbb{N}$ there is some $S \in \mathfrak{S}$ such that
\[
  \kappa(\diamond_k Y)\subseteq S.
\]

A \emph{$\frakZ$--controlled retractive space relative $W$} is a $\frakZ$--controlled $G$--CW--complex  $(Y,\kappa)$ relative $W$ together with a $G$--equivariant retraction $r \colon Y \to W$, ie.~a left inverse to the structural inclusion $W \hookrightarrow Y$. The $\frakZ$--controlled retractive spaces relative $W$ form a category $\cR^G(W,\frakZ)$ in which \emph{morphisms} are $\frakZ$--controlled maps which additionally respect the chosen retractions.
  
The category of controlled $G$--CW--complexes (relative $W$) and controlled maps admits a notion of \textit{controlled homotopies}, see \cite[Definition~2.5]{UW} via the objects $(Y \leftthreetimes [0,1], \kappa \circ pr_Y)$, where $Y \leftthreetimes [0,1]$ denotes the reduced product which identifies $W \times [0,1] \subseteq Y \times [0,1]$ to a single copy of $W$ and $pr_Y: \cells Y \leftthreetimes [0,1] \rightarrow \diamond Y$ is the canonical projection.
In particular, we obtain a notion of \emph{controlled homotopy equivalence} (or \emph{$h$--equivalence}).
  
A $\frakZ$--controlled retractive space $(Y, \kappa)$ is called \textit{finite} if it is finite-dimensional, the image of $Y \backslash W$ under the retraction meets the orbits of only finitely many path components of $W$ and for each $z \in Z$ there is some open neighborhood $U$ of $z$ such that $\kappa^{-1}(U)$ is finite, see \cite[Definition~3.3]{UW}.
  
A $\frakZ$--controlled retractive space $(Y, \kappa)$ is called \textit{finitely dominated}, if there are a finite $\frakZ$--controlled, retractive space $D$, a morphism $p \colon D \rightarrow Y$ and a $\frakZ$--controlled map $i \colon Y \rightarrow D$ such that $p \circ i$ is controlled homotopic to $\id_Y$.
  
The finite, respectively finitely dominated, $\frakZ$--controlled retractive spaces form full subcategories $\cR^G_f(W,\frakZ) \subset \cR^G_{fd}(W,\frakZ) \subset \cR^G(W,\frakZ)$.
All three of these categories support a Waldhausen category structure in which inclusions of $G$--invariant subcomplexes up to isomorphism are the cofibrations and controlled homotopy equivalences are the weak equivalences, see \cite[Corollary~3.22]{UW}.
  
Let $X$ be a $G$--CW--complex and let $M$ be a metric space with free, isometric $G$--action. Define $\frakC_{bdd}(M)$ to be the collection of all subsets $C \subset M \times M$ which are of the form
\begin{equation*}
 C = \{ (m,m') \in M \times M \mid d(m,m') \leq \alpha \}
\end{equation*}
for some $\alpha \geq 0$. Define further $\frakC_{Gcc}(X)$ to be the collection of all $C \subset (X \times [1,\infty[) \times (X \times [1,\infty[)$ which satisfy the following:
\begin{enumerate}
 \item For every $x \in X$ and every $G_x$--invariant open neighborhood $U$ of $(x,\infty)$ in $X \times [1,\infty]$, there exists a $G_x$--invariant open neighborhood $V \subset U$ of $(x,\infty)$ such that $(((X \times [1,\infty[) \setminus U) \times V) \cap C = \emptyset$.
 \item Let $p_{[1,\infty[} \colon X \times [1,\infty[ \to [1,\infty[$ be the projection map. Equip $[1,\infty[$ with the Euclidean metric. Then there exists some $B \in \frakC_{bdd}([1,\infty[)$ such that $C \subset p^{-1}_{[1,\infty[}(B)$.
 \item $C$ is symmetric, $G$--invariant and contains the diagonal.
\end{enumerate}
Next define $\frakC(M,X)$: Let $p_M \colon M \times X \times [1,\infty[ \to M$ and $p_{X \times [1,\infty[} \colon M \times X \times [1,\infty[ \to X \times [1,\infty[$ denote the projection maps. Then $\frakC(M,X)$ is the collection of all subsets $C \subset (M \times X \times [1,\infty[)^2$ which are of the form
\begin{equation*}
 C = p_M^{-1}(B) \cap p_{X \times [1,\infty[}^{-1}(C')    
\end{equation*}
for some $B \in \frakC_{bdd}(M)$ and $C' \in \frakC_{Gcc}(X)$.
  
Finally, define $\frakS(M,X)$ to be the collection of all subsets $S \subset M \times X \times [1,\infty[$ which are of the form $S = K \times [1,\infty[$ for some $G$--compact subset $K \subset M \times X$.
  
All these data combine to a coarse structure
\begin{equation*}
 \JJ(M,X) := (M \times X \times [1,\infty[, \frakC(M,X), \frakS(M,X))
\end{equation*}
which serves to define the ``obstruction category" $\cR^G_f(W,\JJ(G,\EGF{G}{\cF})),h)$, cf.~\cite[Example~2.2 and Definition~6.1]{UW}. The spectrum $\FF(G,W,\EGF{G}{\cF})$ alluded to above is the non-connective $K$--theory spectrum of $\cR^G_f(W,\JJ(G,\EGF{G}{\cF}))$ with respect to the $h$--equivalences, cf.~\cite[Section~5]{UW}. By \cite[Corollary~6.11]{UW}, a group $G$ satisfies the Farrell--Jones Conjecture with coefficients in $A$--theory with respect to $\cF$ if and only if $\FF(G,W,\EGF{G}{\cF})$ is weakly contractible for every free $G$--CW--complex $W$.

Suppose now that $G$ is a DFHJ-group. By definition, there exists some $N \in \NN$ such that for every $n \in \NN$ there is a homomorphism $\pi_n \colon G \to F_n$ to a finite group with the property that for every Dress subgroup $D \leq F_n$ there exist
\begin{enumerate}
 \item a compact, contractible metric space $X_{n,D}$ such that for every $\epsilon > 0$ there is an $\epsilon$--controlled domination of $X_{n,D}$ by an at most $N$--dimensional finite simplicial complex;
 \item a homotopy coherent $G$--action $\Gamma_{n,D}$ on $X_{n,D}$;
 \item a $\pi_n^{-1}(D)$--simplicial complex $\Sigma_{n,D}$ of dimension at most $N$ whose isotropy is contained in $\cF$;
 \item a $\pi_n^{-1}(D)$--equivariant map $f_{n,D} \colon G \times X_{n,D} \to \Sigma_{n,D}$ such that
 \begin{itemize}
  \item for all $g \in G$, $x \in X_D$ and $s \in S^n$
   \[
    d^{l^1}\big( f_{n,D}(g,x),f_{n,D}(gs^{-1},\Gamma_D(s,x)) \big) \leq \frac{1}{n},
   \]
  \item for all $g \in G$, $x \in X_D$ and $s_0,\ldots,s_n \in S^n$
   \[
    \diam\big\{ f_{n,D}(g,\Gamma_{n,D}(s_n,t_n,\ldots,s_0,x)) \,\big|\, t_1,\ldots,t_n \in [0,1] \big\} \leq \frac{2}{n}.
   \]
 \end{itemize}
\end{enumerate}
Assume we have chosen all of this. Then the proof is organized around the following diagram, in which we abbreviate $E := \EGF{G}{\cF}$ (further explanations follow below):
\begin{equation}\label{eq:outline.of.proof}
 \begin{tikzpicture}
    \matrix (m) [matrix of math nodes, column sep=2em, row sep=2em, text depth=.5em, text height=1em, ampersand replacement=\&]
	{(\cR^G_f(W, \JJ(G,E)),h)			                     \& (\cR^G_{fd}(W, \JJ(G,E)),h)                     \\
	 (\cR^G_f(W, \JJ((G)_n,E)),h)                        \& (\cR^G_{fd}(W, \JJ((G)_n,E)),h^{fin})         \\
	 (\cR^G_f(W, \JJ((S_n \times G)_n,E)), h)                     \&                                               \\
	 (\cR^G_{fd}(W, \JJ((X_n \times G)_n,E)), h^{fin}) \& (\cR_{fd}^G(W, \JJ((\Sigma_n \times G)_n, E)), h^{fin})\\};
	\path[->]
     (m-1-1) edge node[above]{$i$} (m-1-2)
     (m-1-1) edge node[left]{$\Delta_f$} (m-2-1)
     (m-1-2) edge node[left]{$\Delta_{fd}$} (m-2-2)
     (m-2-1) edge node[above]{$j$} (m-2-2)
     (m-3-1) edge node[right]{$P_S$} (m-2-1)
     (m-4-1) edge node[below right]{$P_X$} (m-2-2) edge node[above]{$F$} (m-4-2)
     (m-4-2) edge node[right]{$P_\Sigma$} (m-2-2);
    \path[dashed][->]
     (m-1-1.180) edge[bend right=30] node[left]{$\trans_1$} (m-3-1.170)
     (m-3-1) edge node[left]{$\trans_2$}  (m-4-1);
   \end{tikzpicture}
\end{equation}
Diagram~\eqref{eq:outline.of.proof} involves some additional notation which we explain first.

Suppose that $(M_n)_n$ is a sequence of metric spaces with a free, isometric $G$--action. Let $X$ be a $G$--CW--complex. Following \cite[Section~7]{UW}, define the coarse structure
\begin{equation*}
 \JJ((M_n)_n,X) := \big( \coprod_n M_n \times X \times [1,\infty[, \frakC((M_n)_n,X), \frakS((M_n)_n,X) \big)
\end{equation*}
as follows: Members of $\frakC((M_n)_n,X)$ are of the form $C = \coprod_n C_n$ with $C_n \in \frakC(M_n,X)$, and we additionally require that $C$ satisfies the \emph{uniform metric control conditon}:
There is some $\alpha > 0$, independent of $n$, such that for all $((m,x,t)$, $(m',x',t')) \in C$ we have $d(m,m') < \alpha$.
Members of $\frakS((M_n)_n,X)$ are sets of the form $T = \coprod_n T_n$ with $T_n \in \frakS(M_n,X)$.
The resulting category $\cR^G(W,\JJ((M_n)_n,X))$ is canonically a subcategory of the product category $\prod_n \cR^G(W,\JJ(M_n,X))$.
  
Some instances of the category $\cR^G(W,\JJ((M_n)_n,X))$ we consider in diagram \eqref{eq:outline.of.proof} come equipped with another notion of weak equivalence:
Let $(Y_n)_n$ be an object of $\cR^G(W, \JJ((M_n)_n,E))$.
For $\nu \in \NN$, we denote by $(-)_{n > \nu}$ the endofunctor which sends $(Y_n)_n$ to the sequence $(\widetilde{Y}_n)_n$ with $\widetilde{Y}_n = \ast$ for $n \leq \nu$ and $\widetilde{Y}_n = Y_n$ for $n > \nu$.
A morphism $(f_n)_n \colon (Y_n)_n \rightarrow (Y_n')_n$ is an \emph{$h^{fin}$--equivalence} if there is some $\nu \in \NN$, such that $(f_n)_{n>\nu}\colon (Y_n)_{n>\nu} \rightarrow (Y_n')_{n>\nu}$ is an $h$--equivalence.

Next, we define the families of metric spaces that we plug into the coarse structure $\JJ(-,E)$.
As a shorthand, we denote the preimage $\pi_n^{-1}(D)$ of any Dress group $D \leq F_n$ by $\overline{D}$.
\begin{enumerate}
 \item The family $(G)_n$ is the constant family in which we equip each component with the word metric on $G$ with respect to $S$.
 \item Let $\dress_n$ denote the family of Dress subgroups of $F_n$. Then define the $G$--space $S_n := \coprod_{D \in \dress_n} G/\overline{D}$. 
  We equip $S_n \times G$ with the diagonal $G$--action and the quasi-metric $d_{S_n}$ given by
  \[
   d_{S_n}((g_1\overline{D}, g_2), (h_1\overline{D'}, h_2)) :=
    \begin{cases}
     d_G(g_2,h_2) & \overline{D} = \overline{D'}, g_1\overline{D} = h_1\overline{D}, \\
     \infty & \text{otherwise.}
    \end{cases}
  \]
 \item The space $X_n$ is defined to be $\coprod_{D \in \dress_n} X_{n,D} \times G/\overline{D}$.
  Define for each $D \in \dress_n$ the constant $\Lambda_{n,D}$ as in \cite[Section~6]{ELPUW}.
  We equip $X_n \times G$ with the $G$--action $\gamma \cdot (x, g_1\overline{D}, g_2) := (x, \gamma g_1\overline{D}, \gamma g_2)$ and the metric $d_{X_n}$ given by
 \begin{align*}
  &d_{X_n}((x, g_1\overline{D}, g_2),(y, h_1\overline{D'}, h_2) \\
  & :=
   \begin{cases}
    d_G(g_2,h_2) + d_{\Gamma_{n,D},S^n,n,\Lambda_{n,D}}((x,g_2), (y,h_2)) & \overline{D} = \overline{D'}, g_1\overline{D} = g_2\overline{D} \\
    \infty & \text{otherwise,} \\
   \end{cases}
 \end{align*}
 where we use the metric $d_{\Gamma_{n,D},S^n,n,\Lambda_{n,D}}$ defined in \cite[Definition~2.9]{Wegner(2015)}.
 \item Finally, $\Sigma_n$ is defined to be the $G$--simplicial complex $\coprod_{D \in \dress_n} G \times_{\overline{D}} \Sigma_{n,D}$, equipped with the metric $n \cdot d^{\ell^1}$, where $d^{\ell^1}$ denotes the $\ell^1$--metric of a simplicial complex.
\end{enumerate}
When crossing one of the above metric spaces with the group $G$, we regard the resulting space as a metric space by equipping it with the sum of the given metric and the word metric on $G$.
This defines all categories appearing in diagram~\eqref{eq:outline.of.proof}.

Let us now define the functors connecting these categories.
The functors $i$ and $j$ are the exact inclusions functors from finite to finitely dominated objects.
The functors $\Delta_f$ and $\Delta_{fd}$ are the diagonal functors sending a given object $Y$ to the constant sequence $(Y)_n$. Note that $j \circ \Delta_f = \Delta_{fd} \circ i$.
The functors $P_S$, $P_X$ and $P_\Sigma$ are induced the projection maps from $S_n \times G$, $X_n \times G$ and $\Sigma_n \times G$ to $G$.
The functor $F$ is induced by the sequence of maps $(f_n \colon X_n \times G \to \Sigma_n \times G)_n$, which we define by
\[
 f_n(x, g_1\overline{D}, g_2) := (g_1, f_{n,D}(g_1^{-1}g_2, x)).
\]
The formula uses secretly the identification $G/\overline{D} \times G  \cong G \times_{\overline{D}} G$.
Using the contracting properties \Cref{def:dfhj.group}~\eqref{it:dfhj.group4}, one checks that the functor $F$ is well-defined, the proof being completely analogous to \cite[Lemma~6.11]{ELPUW}. Moreover, $P_X = P_\Sigma \circ F$.

We make the following claims:

\begin{proposition}\label{prop:properties.of.main.diagram}
\ 
\begin{enumerate}
 \item\label{it:properties.of.main.diagram1} After applying $K$--theory, the dashed arrow $\trans_1$ exists such that $K_m(\Delta_f) = K_m(P_S) \circ \trans_1$.
 \item\label{it:properties.of.main.diagram2} After applying $K$--theory, the dashed arrow $\trans_2$ exists such that $K_m(j \circ P_S) = K_m(P_X) \circ \trans_2$.
 \item\label{it:properties.of.main.diagram3} The $K$--theory of $(\cR_{fd}^G(W, \JJ((\Sigma_n \times G)_n, E)), h^{fin})$ is trivial.
 \item\label{it:properties.of.main.diagram4} $K_m(\Delta_{fd} \circ i)$ is injective for all $m$.
\end{enumerate}
\end{proposition}

\Cref{thm:afjc.for.dfhjgroups} follows from \Cref{prop:properties.of.main.diagram} by an easy diagram chase.

\begin{proof}[Proof of \Cref{prop:properties.of.main.diagram}]
 Claim~\eqref{it:properties.of.main.diagram1} is an immediate consequence of \cite[Proposition~9.2]{UW}.
 Claim~\eqref{it:properties.of.main.diagram3} is established in \cite[Section~10]{UW}.
 Claim~\eqref{it:properties.of.main.diagram4} is \cite[Lemma~6.12]{ELPUW}.
 So all that is left to show is claim~\eqref{it:properties.of.main.diagram2}.

 The map $\trans_2$ arises as a slight modification of the transfer constructed in \cite[Section~7]{ELPUW}, whose notation we will also use in the following discussion.
 
 Let $\cR^G_f(W,\JJ((S_n \times G)_n, E))_{\alpha,d}$ denote the subcategory of $\cR^G_f(W,\JJ((S_n \times G)_n, E))$ containing only those objects $(Y_n,\kappa_n)_n$ such that $Y_n$ has dimension at most $d$ and is $\alpha$--controlled over $S_n \times G$, together with morphisms $(\phi_n \colon (Y_n,\kappa_n) \to (Y_n',\kappa_n'))_n$ which are \emph{cellwise $0$--controlled} in the following sense:
 Each $\phi_n$ is a regular map (ie. it maps open cells onto open cells), and for each cell $c \in \cells Y_n$, we have $\kappa_n'(\phi_n(c)) = \kappa_n(c)$. Note that such morphisms automatically satisfy the uniform metric control condition.
 
 Arguing as in \cite[Section~7.1]{ELPUW}, we observe that it suffices to construct compatible transfers on each $\cR^G_f(W,\JJ((S_n \times G)_n, E))_{\alpha,d}$ individually.
 
 Let $(Y_n,\kappa_n)_n$ be an object in $\cR^G_f(W,\JJ((S_n \times G)_n, E))_{\alpha,d}$.
 By the definition of the metric $d_{S_n}$, the complex $Y_n$ decomposes $G$--equivariantly as $Y_n = \coprod_{D \in \dress_n} Y_{n,D}$, with $Y_{n,D}$ living over the metric component $G/\pi_n^{-1}(D) \times G$.
 Let $\kappa_{n,D}$ denote the restriction of $\kappa_n$ to the set of cells of $Y_{n,D}$.
 Then define
 \[
  \trans^{\alpha,d}_n(Y_n) := \coprod_{D \in \dress_n} \trans^{\alpha,d}_{X_{n,D}}(Y_{n,D}),
 \]
 cf.~\cite[Definition~7.9]{ELPUW}.
 The control map $\trans^{\alpha,d}_n(\kappa_n)$ of $\trans^{\alpha,d}_n(Y_n)$ is defined as in {\it loc.~cit.} (formula directly before Lemma~7.10), replacing $G$ by $S_n \times G$.
 Then the obvious analog of \cite[Lemma~7.10]{ELPUW} holds, so that
 \[
  \trans^{\alpha,d}((Y_n,\kappa_n)_n) := (\trans^{\alpha,d}_n(Y_n),\trans^{\alpha,d}_n(\kappa_n))_n
 \]
 is indeed an object in $\cR^G_{fd}(W,\JJ(X_n \times G)_n,E))$.
 By the obvious analog of \cite[Lemma~7.11]{ELPUW}, $\trans^{\alpha,d}$ defines a functor 
\[ 
  \trans^{\alpha,d} \colon \cR^G_f(W,\JJ((S_n \times G)_n, E))_{\alpha,d} \to \cR^G_{fd}(W,\JJ(X_n \times G)_n,E)).
 \]
 Since we leave the $S_n \times G \times E \times [1,\infty[$--component of each $\kappa_n$ unchanged, the rest of \cite[Section~7]{ELPUW} carries over to show the existence of the map $\trans_2$, and thus claim~\eqref{it:properties.of.main.diagram2}.
\end{proof}

\begin{remark}\label{rem:deloopings}
 In fact, the discussion we have given so far only establishes the vanishing of $K_m(\cR^G_f(W,\JJ(E)),h)$ for $m > 0$.
 In order to show vanishing in all degrees, we need to consider appropriate deloopings constructed by introducing another metric coordinate $\RR^k$. Since this coordinate remains unchanged throughout, the previous discussion applies verbatim. Cf.~also \cite[Section~9]{UW} and the discussion in Section~6 of \cite{ELPUW}.
\end{remark}

\section{Proof of the main theorem}

As in \cite[Section~3]{Wegner(2015)}, the first step in proving \Cref{maintheorem} lies in reducing the general theorem to some special cases. For any non-zero algebraic number $w$, set $G_w := \ZZ[w,w^{-1}] \rtimes_{\cdot w} \ZZ$.

\begin{lemma}\label{lem:reduction.to.gw}
 If $G_w$ satisfies the $A$--theoretic Farrell--Jones Conjecture with coefficients and finite wreath products for every non-zero algebraic number $w$, then so does every virtually solvable group.
\end{lemma}
\begin{proof}
 We claim that the arguments in \cite[Section~3]{Wegner(2015)} carry over to $A$--theory.
 Indeed, the argument relies only on the following statements about the Farrell--Jones Conjecture with coefficients and finite wreath products:
 \begin{enumerate}
   \item The class of groups satisfying the conjecture has the following closure properties \cite[Theorem~1.1(ii)]{ELPUW}:
    \begin{itemize}
     \item If a group satisfies the conjecture, so does every subgroup.
     \item If two groups satisfy the conjecture, so do their direct and free products.
     \item If $\{ G_i \}_{i \in I}$ is a directed system of groups satisfying the conjecture, so does the colimit.
     \item If $p \colon G \twoheadrightarrow Q$ is an epimorphism, and $Q$ as well as every preimage $p^{-1}(C)$ of virtually cyclic subgroups of $Q$ satisfy the conjecture, so does $G$.
    \end{itemize}
   \item The following groups satisfy the conjecture:
    \begin{itemize}
     \item Semidirect products $A \rtimes \ZZ$ with $A$ torsion abelian: This case follows from the case of hyperbolic groups \cite[Theorem~1.1(i)]{ELPUW}, cf.~\cite[Lemma~4.1]{Farrell-Linnell(2003)}.
     \item The wreath product $\ZZ \wr \ZZ$: This is, for example, a directed colimit of $\mathrm{CAT}(0)$--groups, and hence satisfies the conjecture by \cite[Theorem~1.1(i)]{ELPUW}. Alternatively, one can argue as in \cite[Lemma~4.3]{Farrell-Linnell(2003)}.
     \item Virtually abelian groups \cite[Corollary~11.11]{UW}, \cite[Theorem~1.1(i)]{ELPUW}.
    \end{itemize}
 \end{enumerate}
 For details, we refer to \cite[Section~3]{Wegner(2015)}.
\end{proof}

If $w$ is a root of unity, $G_w$ is a virtually abelian group (cf.~\cite[Lemma~5.32]{Wegner(2015)}) and satisfies the $A$--theoretic Farrell Jones Conjecture with coefficients and finite wreath products by \cite[Corollary~11.11]{UW}.
So we may assume that $w$ is not a root of unity in the sequel.

We recall some notation from \cite[Section~5]{Wegner(2015)}.
In what follows, we fix a non-zero algebraic number $w$ which is not a root of unity.
Let $\cO$ be the ring of integers in $\QQ(w)$.
Define the ring $\cO_w$ to be
\[
 \cO_w := \{ x \in \QQ(w) \mid v_\frakp(x) \geq 0 \text{ for all prime ideals $\frakp \subset \cO$ with $v_\frakp(w) = 0$.} \},
\]
so that $\cO \subseteq \cO_w$ and $w$, $w^{-1} \in \cO_w$.

For $s \in \NN$ we define $t_w(s) \geq 0$ to be the number determined by
\[
 t_w(s)\ZZ = \{ z \in \ZZ \mid w^z \equiv 1 \mod s\cO_w \}.
\]

\begin{lemma}\label{lem:images.of.dress.groups}
Let $q_1$, $q_2$ be prime numbers satisfying $q_1 \neq q_2$ and $v_\frakp(w) = 0$ for all prime factors $\frakp$ of $q_1$ or $q_2$ in $\cO$. Let $m_1, m_2$ be natural numbers.

Consider the finite group $F := ( \cO_w / q_1^{m_1} q_2^{m_2} \cO_w ) \rtimes \ZZ / t_w(q_1^{m_1} q_2^{m_2}) \ZZ$.

For every Dress group $D \leq F$, there exists $i \in \{ 1,2 \}$ such that the image of $D$ under the canonical projection
$\eta_i \colon F \twoheadrightarrow \cO_w / q_i^{m_i} \cO_w \rtimes \ZZ / t_w(q_i^{m_i}) \ZZ$ is hyperelementary.
\end{lemma}
\begin{proof}
 Let $D$ be a Dress subgroup of $F$. Then $D$ fits into a normal series $P \unlhd H \unlhd D$ such that $P$ is a $p$--group, $D/H$ is a $p'$--group, $P$ is normal in $D$ and $\abs{H/P}$ is coprime to both $p$ and $p'$ \cite[Lemma~5.1]{Winges(2015)}.

 The prime $p$ cannot be $q_1$ and $q_2$ at the same time; without loss of generality, assume that $p \neq q_1$.
 Set $t := t_w(q_1^{m_1}q_2^{m_2})$ and $t_1 := t_w(q_1^{m_1})$.
 Consider the normal subgroup $N := q_1^{m_1} \cO_w / q_1^{m_1}q_2^{m_2} \cO_w \rtimes t_1 \ZZ / t\ZZ$ and let $\eta_1$ denote the projection map
 \[
  \eta_1 \colon F \twoheadrightarrow F / N \cong \cO_w / q_1^{m_1} \cO_w \rtimes \ZZ / t_1\ZZ.
 \]
 Then $\eta_1(P) \cap \cO_w / q_1^{m_1} \cO_w = \{ 0 \}$ since the latter is a $q_1$--group and $p \neq q_1$.
 Hence, $\eta_1(P)$ is mapped isomorphically to a subgroup of $\ZZ / t_1\ZZ$ by the projection map $\cO_w / q_1^{m_1} \cO_w \rtimes \ZZ / t_1\ZZ \twoheadrightarrow \ZZ/t_1\ZZ$. So $\eta_1(P)$ is cyclic.
 Since $p$ is coprime to $\abs{ H/P }$ and $H/P$ is cyclic, the image $\eta_1(H)$ is also cyclic.
 It follows that $\eta_1(D)$ is hyperelementary.
\end{proof}

\begin{proposition}\label{prop:gw.is.dfhj}
 Let $w \neq 0$ be an algebraic number which is no root of unity.
 Then $G_w = \ZZ[w, w^{-1}] \rtimes \ZZ$ is a DFHJ--group with respect to the family of virtually abelian subgroups.
\end{proposition}

\begin{proof}
Let $N$ be the natural number determined by \cite[Proposition~5.26]{Wegner(2015)}. Let $S \subseteq G_w$ be a finite, symmetric generating set containing the trivial element.

In the proof of \cite[Proposition~5.33]{Wegner(2015)}, it is shown that for every $n \in \N$ and for every sufficiently large prime number $q$ (depending on $n$) there is a natural number $m \in N$ such that for every hyperelementary subgroup
\[
H \leq F_n := \cO_w / q^m \cO_w \rtimes \ZZ / t_w(q^m)\ZZ
\]
there exist
\begin{enumerate}
  \item a compact, contractible metric space $X_{n,H}$ such that for every $\epsilon > 0$ there is an $\epsilon$--controlled domination of $X_{n,H}$ by an at most $N$--dimensional finite simplicial complex;\footnote{In the proof of \cite[Proposition~5.33]{Wegner(2015)} the space $X_{n,H}$ is denoted by $X_w^R$.}
  \item a homotopy coherent $G_w$--action $\Psi_{n,H}$ on $X_{n,H}$;
  \item a positive real number $\Lambda_{n,H}$;
  \item a $\alpha_n^{-1}(H)$--simplicial complex $E_{n,H}$ of dimension at most $N$ whose isotropy groups are virtually cyclic or abelian;
  \item a $\alpha_n^{-1}(H)$--equivariant map $f_{n,H} \colon G_w \times X_{n,H} \to E_{n,H}$ such that
      \[
      n \cdot d^{l^1}\big( f_{n,H}(g,x),f_{n,H}(h,y) \big) \leq d_{\Psi_{n,H},S^n,n,\Lambda_{n,H}}\big( (g,x),(h,y) \big)
      \]
      for all $(g,x), (h,y) \in G_w \times X_{n,H}$ with $h^{-1}g \in S^n$.
\end{enumerate}
Here, $\alpha_n \colon G_w \to F_n$ denotes the composition of the inclusion $G_w \hookrightarrow \cO_w \rtimes \ZZ$ with the quotient map $\cO_w \rtimes \ZZ \twoheadrightarrow F_n$. The metric $d_{\Psi_{n,H},S^n,n,\Lambda_{n,H}}$ on $G_w \times X_{n,H}$ is defined in \cite[Definition~2.9]{Wegner(2015)}. It has the property
\[
d_{\Psi_{n,H},S^n,n,\Lambda_{n,H}}\big( (g,x),(g (s_n \cdots s_0)^{-1},\Psi_{n,H}(s_n,t_n,\ldots,s_0,x)) \big) \leq 1
\]
for all $g \in G_w$, $x \in X_{n,H}$ and $s_0,\ldots,s_n \in S^n$.
Hence,
\[
d^{l^1}\big( f_{n,H}(g,x),f_{n,H}(gs^{-1},\Psi_{n,H}(s,x)) \big) \leq \frac{1}{n}
\]
for all $g \in G_w$, $x \in X_{n,H}$ and $s \in S^n$, and
\[
\diam\big\{ f_{n,H}(g,\Psi_{n,H}(s_n,t_n,\ldots,s_0,x)) \,\big|\, t_1,\ldots,t_n \in [0,1] \big\} \leq \frac{2}{n}
\]
for all $g \in G_w$, $x \in X_{n,H}$ and $s_0,\ldots,s_n \in S^n$.

Now let us come to the actual proof. For a given $n \in \N$ we choose two distinct (large) prime numbers $q_1$, $q_2$ with appropriate natural numbers $m_1, m_2 \in N$ (as described above).
Consider the finite group
\[
F := \cO_w / q_1^{m_1} q_2^{m_2} \cO_w \rtimes \ZZ / t_w(q_1^{m_1} q_2^{m_2})\ZZ.
\]
Let $D \leq F$ be a Dress subgroup. By \Cref{lem:images.of.dress.groups}, there exists $i \in \{1,2\}$ such that $\eta_i(D)$ is hyperelementary. We have a finite group $F_n := \cO_w / q_i^{m_1} \cO_w \rtimes \ZZ / t_w(q_i^{m_i})\ZZ = \im(\eta_i)$ with a hyperelementary subgroup $H := \eta_i(D) \leq F_n$. As mentioned at the beginning of the proof, we obtain a homotopy coherent $G_w$--action $\Gamma_{n,H}$ on a metric space $X_{n,H}$, an $\alpha_n^{-1}(H)$--simplicial complex $E_{n,H}$ and an $\alpha_n^{-1}(H)$--equivariant map $f_{n,H}$ with the properties described above. We define $\pi \colon G_w \to F$ as the composition of the inclusion $G_w \hookrightarrow \cO_w \rtimes \ZZ$ with the quotient map $\cO_w \rtimes \ZZ \twoheadrightarrow F$. Then $\pi^{-1}(D)$ is a subgroup of $\alpha_n^{-1}(H)$. We finally set $X_D := X_{n,H}$, $\Gamma_D := \Psi_{n,H}$, $\Sigma_D := E_{n,H}$, $f_D := f_{n,H}$.
\end{proof}

Since virtually abelian groups satisfy the $A$--theoretic Farrell--Jones Conjecture with coefficients and finite wreath products, \Cref{maintheorem} follows from \Cref{lem:reduction.to.gw}, \Cref{prop:gw.is.dfhj} and \Cref{thm:afjc.for.dfhjgroups} together with the Transititvity Principle \cite[Proposition~11.2]{UW} in view of the following:

\begin{lemma}\label{lem:virtually.dfhj}
 Suppose that $G$ is a DFHJ-group with respect to the family of all subgroups which satisfy the $A$--theoretic Farrell--Jones Conjecture with coefficients and finite wreath products.
 Let $F$ be a finite group.

 Then $G \wr F$ is a DFHJ-group with respect to the family of all subgroups which satisfy the $A$--theoretic Farrell--Jones Conjecture with coefficients and finite wreath products.
\end{lemma}
\begin{proof}
 The proof is analogous to that of \cite[Lemma~4.3]{Wegner(2015)}, replacing ``hyperelementary" by ``Dress"  and using the fact that the collection of Dress groups is also closed under taking subgroups and quotients.\end{proof}

\bibliographystyle{alpha}
\bibliography{AFJC_solvable_bib}

\newcommand{\etalchar}[1]{$^{#1}$}
\begin{thebibliography}{KLR16}

\bibitem[BL12a]{Bartels-Lueck(2012)}
A.~Bartels and W.~L{\"u}ck.
\newblock The {F}arrell-{H}siang method revisited.
\newblock {\em Math. Ann.}, 354(1):209--226, 2012.

\bibitem[BL12b]{Bartels-Lueck(2012b)}
Arthur Bartels and Wolfgang L{\"u}ck.
\newblock The {B}orel conjecture for hyperbolic and {${\rm CAT}(0)$}-groups.
\newblock {\em Ann. of Math. (2)}, 175(2):631--689, 2012.

\bibitem[DL98]{Davis-Lueck(1998)}
James~F. Davis and Wolfgang L{\"u}ck.
\newblock Spaces over a category and assembly maps in isomorphism conjectures
  in {$K$}- and {$L$}-theory.
\newblock {\em $K$-Theory}, 15(3):201--252, 1998.

\bibitem[ELP{\etalchar{+}}]{ELPUW}
Nils-Edvin Enkelmann, Wolfgang L\"uck, Malte Pieper, Mark Ullmann, and
  Christoph Winges.
\newblock On the {F}arrell--{J}ones {C}onjecture for {W}aldhausen's
  ${A}$--theory.
\newblock arXiv:1607.06395.

\bibitem[FL03]{Farrell-Linnell(2003)}
F.~Thomas Farrell and Peter~A. Linnell.
\newblock {$K$}-theory of solvable groups.
\newblock {\em Proc. London Math. Soc. (3)}, 87(2):309--336, 2003.

\bibitem[KLR16]{KLR(2016)}
Holger Kammeyer, Wolfgang L{\"u}ck, and Henrik R{\"u}ping.
\newblock The {F}arrell--{J}ones conjecture for arbitrary lattices in virtually
  connected {L}ie groups.
\newblock {\em Geom. Topol.}, 20(3):1275--1287, 2016.

\bibitem[L{\"u}c05]{Lueck(2005)}
Wolfgang L{\"u}ck.
\newblock Survey on classifying spaces for families of subgroups.
\newblock In {\em Infinite groups: geometric, combinatorial and dynamical
  aspects}, volume 248 of {\em Progr. Math.}, pages 269--322. Birkh\"auser,
  Basel, 2005.

\bibitem[R{\"u}p16]{Rueping(2016)}
H.~R{\"u}ping.
\newblock The {F}arrell-{J}ones conjecture for {$S$}-arithmetic groups.
\newblock {\em J. Topol.}, 9(1):51--90, 2016.

\bibitem[UW]{UW}
Mark Ullmann and Christoph Winges.
\newblock On the {F}arrell--{J}ones {C}onjecture for algebraic {$K$}--theory of
  spaces: the {F}arrell--{H}siang method.
\newblock arXiv:1509:07363.

\bibitem[Weg12]{Wegner(2012)}
Christian Wegner.
\newblock The {$K$}-theoretic {F}arrell-{J}ones conjecture for {CAT}(0)-groups.
\newblock {\em Proc. Amer. Math. Soc.}, 140(3):779--793, 2012.

\bibitem[Weg15]{Wegner(2015)}
Christian Wegner.
\newblock The {F}arrell-{J}ones conjecture for virtually solvable groups.
\newblock {\em J. Topol.}, 8(4):975--1016, 2015.

\bibitem[Win15]{Winges(2015)}
Christoph Winges.
\newblock On the transfer reducibility of certain {F}arrell--{H}siang groups.
\newblock {\em Algebr. Geom. Topol.}, 15(5):2921--2948, 2015.

\end{thebibliography}

\end{document}